\documentclass[a4paper]{article}
\usepackage[utf8]{inputenc}
\usepackage{amsmath}
\usepackage{amsthm}
\usepackage{amsfonts}
\usepackage{amssymb}
\usepackage{graphicx}
\usepackage{authblk}

\usepackage{hyperref}
\usepackage{cite}
\usepackage{array}
\usepackage{tabulary}
\usepackage{multicol}
\usepackage{enumerate}
\usepackage{multido}
\usepackage{multirow}
\usepackage{latexsym, amsbsy}

\newtheorem{theorem}{Theorem}[section]

\newtheorem{example}{Example}[section]
\newtheorem{corollary}{Corollary}[section]

\newtheorem{remark}{Remark}[section]

\begin{document}
\title{{\bf Some results on number theory and differential equations}}
\author{B. M. Cerna Magui\~na and D.D. Lujerio Garcia.
\\Facultad de Ciencias, Departamento Acad\'emico de Matem\'atica   
\\Universidad Nacional Santiago Ant\'unez de Mayolo
\\Campus Shancayan, Av. Centenario 200, Huaraz, Per\'u 
\\email:bcernam@unasam.edu.pe}

\maketitle
\begin{abstract}
In this work we state a Theorem on number theory and apply it  to solve some ordinary and partial differential  equations.
\end{abstract}
\vspace{0.5 cm}
\newpage
\section{Introduction}
In \cite{LS}, it has been introduced a technique to find integer solutions for quadratic polynomials, in two variables, which represent a given  natural number.  This thechnique  will be used to find certain solutions of some ordinary and partial differential equations. In addition, we  discuss the Fermat's last equation.   We also write an even number as the sum of two natural numbers and show that under certain conditions these numbers are primes. The primality  of a number is subjected to give the general integer solution of the equation $(x-n)^{2}+(y-n)^{2}=2n^{2}$, where $n$ is a natural number.

\section{Some results on number theory and differential equations}
\begin{theorem}\label{th1}
Let $M$ be an even number endig in 6, then there exist $a,b\in \mathbb{N}$, where $a$ is divisor of $M$, and $a$ and  $b$ are relatively primes with $a>b$ and $M\displaystyle\frac{(a+b)}{2a}\in \mathbb{N}$, $M\displaystyle\frac{(a-b)}{2a}\in \mathbb{N}$ such that: 	$M=\displaystyle\frac{M(a+b)}{2 a}+\frac{M(a-b)}{2 a}$.																			
\end{theorem}
\begin{proof}
Let
\begin{eqnarray} \label{M1}
M=(K-\alpha) + (K+\alpha)=(10k_{1}+3)+(10k_{2}+3)
\end{eqnarray}
where $k_{2}=\displaystyle\frac{K-3+\alpha}{10}$, $k_{1}=\displaystyle\frac{K-3-\alpha}{10}$, $\alpha$, $k_{1}$, $k_{2}$ are natural numbers.

Let us define the next linear functional
\begin{eqnarray}\label{M2}
F(x,y)=(10k_{1}+3)x+(10k_{2}+3)y. 
\end{eqnarray}
So, $\ker F= \{(-(10k_{2}+3),10k_{1}+3 )\}$, $\{\ker F\}^{\bot}= \{(10k_{1}+3 ,10k_{2}+3 )\}$.
We have
\begin{eqnarray}\label{M3}
F(1,1)=M.
\end{eqnarray}
Moreover
\begin{eqnarray}\label{M4}
(1,1)=\lambda_{1} (-(10k_{2}+3),(10k_{1}+3))+\lambda_{2}((10k_{1}+3), 10k_{2}+3).
\end{eqnarray} 
From (\ref{M2}) and  (\ref{M4}) we have
\begin{eqnarray} \label{M5}
M=\lambda_{2}[(10k_{1}+3)^{2}+(10k_{2}+3)^{2}].
\end{eqnarray}
From (\ref{M5}) and (\ref{M1}) it follows
\begin{eqnarray*} \label{M6}
M=\lambda_{2} \left[(K-\alpha)^{2}+(K+\alpha)^{2} \right]=\lambda_{2}[2K^{2}+2\alpha^{2}]
\end{eqnarray*}
from (\ref{M5}) it follows that $\lambda_{2}\in \mathbb{Q}$. Let $\lambda_{2}=\frac{m}{n}$, where $m$ and $n$ are relatively primes, $m<n$.
From this last relation one has 
\begin{eqnarray}\label{M7}
\vert Mm-n\vert=\sqrt{n^{2}-4\alpha^{2}m^{2}}.
\end{eqnarray}
Using the Fermat's Theorem for the power $2$ one has $n=a^{2}+b^{2}$, $w=a^{2}-b^{2}$, $a>b$, $\alpha m=ab$; then, using these ralationships into (\ref{M7}) one has
\begin{eqnarray}\label{M8}
\alpha=\frac{Mb}{2a}\,\, \mbox{o} \,\, \alpha =\frac{Ma}{2b}
\end{eqnarray}
 From  (\ref{M2}) one has
\begin{eqnarray} \label{M9}
F(1,0)=10 k_{1}+3
\end{eqnarray}
From (\ref{M9}) we see that $10k_{1}+3 $ is a number ending in 3. Let us assume that this number in not prime. Then 
\begin{eqnarray}\label{asterisco1}
10k_{1}+3=(10A+7)(10B+9)\,\, \mbox{or}\,\,10k_{1}+3=(10A+3)(10B+1)
\end{eqnarray}
where $(A,B)\in \mathbb{N}\times \mathbb{N}$.

So, one has 
\begin{eqnarray}\label{M10}
(1,0)=r_{1}(-(10k_{2}+3), (10k_{1}+3))+r_{2}(10k_{1}+3, 10k_{2}+3),
\end{eqnarray}
from (\ref{M9}) y (\ref{M10})
\begin{eqnarray}\label{M11}
10k_{1}+3=r_{2}((10k_{1}+3)^{2}+(10k_{2}+3)^{2}).
\end{eqnarray}
From the relation (\ref{M2}) one has
\begin{eqnarray} \label{M12}
F(0,1)=10k_{2}+3
\end{eqnarray}
From (\ref{M12}) one has $10k_{2}+3$ is a number ending in 3; let us assume that this in not a prime number. Then 
\begin{eqnarray}\label{asterisco2}
10k_{2}+3=(10C+3)(10D+1)\,\,\mbox{or}\,\, 10k_{2}+3=(10C+7)(10D+9)
\end{eqnarray}
where $(C,D)\in \mathbb{N}\times \mathbb{N}.$

From the relation (\ref{M12}) and the relationship 

$(0,1)=t_{1}(-(10k_{2}+3), (10k_{1}+3))+t_{2}(10k_{1}+3,10k_{2}+3)$, one gets
\begin{eqnarray}\label{M13}
10k_{2}+3=t_{2}[(10k_{1}+3)^{2}+(10k_{2}+3)^{2}].
\end{eqnarray}
Since (\ref{M1}) $10k_{1}+3=K-\alpha$ y $10k_{2}+3=K+\alpha$, from  (\ref{M11}) and  (\ref{M12}) one has
\begin{eqnarray}\label{M14}
M&=&(r_{2}+t_{2})[(10k_{1}+3)^{2}+(10k_{2}+3)^{2}]=(r_{2}+t_{2})[\frac{M^{2}}{2}+2\alpha^{2}] \\ 
2\alpha &=& (t_{2}-r_{2})[(10k_{1}+3)^{2}+(10k_{2}+3)^{2}]=(t_{2}-r_{2})[\frac{M^{2}}{2}+2\alpha^{2}]. \nonumber
\end{eqnarray}
From (\ref{M14}) and  (\ref{M8}) one can get the following
\begin{eqnarray}\label{M15}
t_{2}=\frac{a(a+b)}{M(a^{2}+b^{2})}\,\, \mbox{and} \,\, r_{2}=\frac{a(a-b)}{M(a^{2}+b^{2})}, a>b.
\end{eqnarray}
From the relations (\ref{M11}), (\ref{M15}) and (\ref{M13}) one gets
\begin{eqnarray}\label{M16}
10k_{1}+3=\frac{(a-b)M}{2a}\,\, \mbox{and} \,\, 10k_{2}+3=\frac{(a+b)M}{2a}, \,\, a>b.
\end{eqnarray}
where $k_{1}=\frac{M(a-b)-6a}{20a}$ and  $k_{2}=\frac{M(a+b)-6a}{20a}$.

The other possibility to write  $M$ is given  by the next equation
\begin{eqnarray}\label{M17}
M=(K-\alpha)+(K+\alpha)=(10k_{1}+9)+(10k_{2}+7)=(10k_{1}+9)+(10k_{2}+7)
\end{eqnarray}

The process is similar to  (\ref{M1}); so,  $k_{1}$, $k_{2}$ are given by

$k_{1}=\displaystyle\frac{K-9-\alpha}{10}$, $k_{2}=\displaystyle\frac{K+\alpha-7}{10}$ o $k_{2}=\displaystyle\frac{K-7-\alpha}{10}$, $k_{1}=\displaystyle\frac{K+\alpha-9}{10}$ donde $\alpha =\displaystyle\frac{Mb}{2a}$, $2K=M$.

Finally,  the last possibility to write $M $ is in the following form 
\begin{eqnarray}\label{M18}
M=5+(10k_{1}+1)=(K-\alpha)+(K+\alpha)
\end{eqnarray}
donde $\alpha=K-5$ y $k_{1}=\frac{K-3}{5}$.
\end{proof}

\begin{corollary}
 $(10k_{1}+3)$ and $(10k_{2}+3)$  are primes if and only if $(a-b)$ and $(a+b)$ are prime numbers. This happens when $2a=M$.
\end{corollary}
\begin{proof} The proof follows by using (\ref{M16}).
\end{proof}
\begin{corollary}
 The relationships between  $\alpha, m, n$ and  $K$ are given by $m=K$ and  $\alpha =\sqrt{n-K^{2}}$, where $K^{2}\leq n \leq 2K^{2}$.
\end{corollary}
\begin{proof} Use the relation (\ref{M6}). For $K$ even, $\alpha$ is odd and  $n$ is odd. For $K$ odd, $\alpha$ is even  and $n$ is odd.
\end{proof}
\begin{remark}\label{r1}
The primality of a natural number is subject to the integer solution of the equation $(x-n)^{2}+(y-n)^{2}=2n^{2}$, where $n\in \mathbb{N}$.
\end{remark}
\begin{corollary}
Let $M=2K$, where  $M$ is a natural number ending in 6. For each $K$ being a natural number ending in 8 we hava that if  $K+\alpha$ and $K-\alpha$ are prime numbers then $\alpha \in \{1, 11, 21, ..., K-7\}$  or  $\{5, 15, 25,..., K-3\}$. If $K$ ends in 3 then $\alpha \in \{0,10,20,...,K-3\}$ or $\alpha \in \{ 6, 16, 26,..., K-7\}$.
\end{corollary}
\begin{proof}
Simply, we use the equations (\ref{M16}) y (\ref{M17}) to obtain $\frac{(K+\alpha-9)^2}{50}\in \mathbb{N}$ or $\frac{(K+\alpha -3)^{2}}{50}\in \mathbb{N}$, $0\leq \alpha \leq K$, and from this the result follows.
\end{proof}
\begin{remark}
Let  $K+\alpha_{0}$ and $K-\alpha_{0}$ be prime numbers satisfying the condition $M=2K$, $K$ ending in  $8$ and  $\alpha_{0}\in \{1, 11,..., K-7\}$.

For $\widetilde{M}=M+10=2\widetilde{K}$ imply $\widetilde{K}=K+5$, $\widetilde{\alpha}_{0}=\alpha_{0}-1$ y $(\widetilde{K}+\widetilde{\alpha}_{0} )+ (\widetilde{K}-\widetilde{\alpha}_{0} )=2\widetilde{K}=\widetilde{M}$, where $\dfrac{(K+\alpha_{0}+1)^{2}}{50}\in \mathbb{N}$.
\end{remark}

If $K$ ends in 8 and  $\alpha_{0}\in \{5,15,..., K-3\}$ one can have $\widetilde{K}=K+5$, $\widetilde{\alpha}_{0}=\alpha_{0}+1 $ y $(\widetilde{K}+\widetilde{\alpha}_{0} )+ (\widetilde{K}-\widetilde{\alpha}_{0} )=2\widetilde{K}=\widetilde{M}$, where $\dfrac{(K+\alpha_{0}-3)^{2}}{50}\in \mathbb{N}$.
A similar result can be obtained if $K$ ends in $3$.

\begin{remark}
If $K-\alpha$ and $K+\alpha$ end in 3 where  $M=2K$, then $K-\alpha$ and $K+\alpha$ possess the represntations given in  (\ref{asterisco1}) and  (\ref{asterisco2}), and if they possess integer solutions using the Theorem  \ref{th1} given in  \cite{LS} we have

$$
\frac{\sqrt{10 K_{1}+3-2 \sqrt{10 K_{1}+3}}-7}{5} \leq A+B \leq 2\frac{(\sqrt{10 K_{1}-31}-7)}{5}
$$ or 
$$\dfrac{\left( \sqrt{\dfrac{110 K_{1}+33}{13}}-1\right)}{5} \leqslant A+B \leq \dfrac{2\left(\sqrt{10 K_{1}+1}-1\right)}{5}$$
This is for the first representation of $K-\alpha$ given in  (\ref{asterisco1}).
\end{remark}
We can obtain similar results if $K+\alpha$ has the  possible representation given in (\ref{asterisco2}).

\subsection{Some applications}
\begin{example} \label{EJEMPLO1} \cite{TP}
Let us consider the Cauchy problem in the interval $[0,\ell]$ 
\begin{eqnarray*}
\frac{d^{2}x}{dt^{2}}+b(t)\frac{dx}{dt}+c(t)x(t)=y(t), \,\, x(0)=\alpha, x'(0)=\beta.
\end{eqnarray*}
It is convenient to write it in the form of the system $\frac{dx}{dt}-u(t)=0,$\newline
\begin{eqnarray}\label{ej1}
\frac{du}{dt}+b(t)u(t)+c(t)x(t)=y(t), \,\, x(0)=\alpha, \,\, u(0)=\beta.
\end{eqnarray}
In order to solve (\ref{ej1}) let us apply the  Theorem (\ref{th1}).

Let $F(x,y)=u'(t)x+b(t)u(t)y$, and assume that $\vert u'(t)\vert \leq M_{1}$,  $\vert b(t)\vert \leq M_{2}$, $\vert u(t)\vert \leq M_{3} $  $\forall t\in [0,\ell]$. Then the operator $F$ is bounded. Then  $\ker F$ is closed. (Even though this is not necessary since every linear operator $F: \mathbb{R}^{n}\longrightarrow \mathbb{R}^{m}$ is bounded).

So, one has $\ker F=\{(-b(t)u(t), u'(t))\}$, $\{\ker F\}^{\perp}=\{(u'(t),-b(t)u(t))\}$
\begin{eqnarray}\label{ej2}
F(1,1)=u'(t)+b(t)u(t)=f(t), \,\, \mbox{where} \,\, f(t)=y(t)-c(t)x(t)
\end{eqnarray}
Since
\begin{eqnarray} \label{ej1_ec3}
(1,1)=\lambda_{1}(-b(t)u(t),u'(t))+\lambda_{2}(u'(t),b(t)u(t)),
\end{eqnarray}
then 
\begin{eqnarray} \label{ej1_ec4}
f(t)&=&\lambda_{2}(u'^{2}(t)+b^{2}(t)u^{2}(t)).
\end{eqnarray}
A parametrization satisfying (\ref{ej1_ec4}) is
\begin{eqnarray}
u'(t)&=&\sqrt{\dfrac{f(t)}{\lambda_{2}(t)}}\sin (\frac{\pi t}{\ell})   \label{ej1_ec5}\\
b(t)u(t)&=&\sqrt{\dfrac{f(t)}{\lambda_{2}(t)}}\cos (\frac{\pi t}{\ell}) \label{ej1_ec6}.
\end{eqnarray}
From the relations  (\ref{ej1_ec5}) and  (\ref{ej1_ec6}) one gets
\begin{eqnarray}\label{ej1_ec7}
\lambda_{2}(t)=\dfrac{1+\sin 2(\frac{\pi t}{\ell})}{f(t)}.
\end{eqnarray}
From the relations (\ref{ej1_ec3}) one gets
\begin{eqnarray}\label{ej1_ec8}
\lambda_{2}(t)=\dfrac{f(t)}{u'^{2}(t)+u^{2}(t)b^{2}(t)}
\end{eqnarray}

From (\ref{ej1_ec7}) and (\ref{ej1_ec8}) one has
\begin{eqnarray}\label{ej1_ec9}
[u'^{2}(t)+ u^{2}(t)b^{2}(t)](1+\sin 2(\frac{\pi t}{\ell}) )=f^{2}(t)
\end{eqnarray}
From (\ref{ej1_ec5}) and  (\ref{ej1_ec6}) one has
\begin{eqnarray}\label{ej1_ec10}
\dfrac{u'(t)}{u(t)}=b(t)\tan (\frac{\pi t}{\ell})
\end{eqnarray}
From (\ref{ej1_ec9}) and (\ref{ej1_ec10})  one has 
\begin{eqnarray}\label{ej1_ec11}
f(t)=u(t)b(t)(1+\tan (\frac{\pi t}{\ell}) ).
\end{eqnarray}
Therefore from  (\ref{ej1_ec10}) one gets that $x(t)=\alpha +\beta \displaystyle\int_{0}^{t} e^{G(t)}d\tau$, where $G(t)=\int _{0}^{\tau}b(t)\tan (\frac{\pi t}{\ell}) dt$, and in addition from  (\ref{ej1_ec11}) one has that 
$$y(t)-\beta e^{G(t)}(1+\tan (\frac{\pi t}{\ell}))b(t)=\left( \alpha +\beta \int_{0}^{t}e^{G(t)}d\tau \right)c(t) $$.
\end{example}
\begin{example} \label{EJEMPLO2} \cite{CHA}
Let  $D$ be a open region and bounded in the plane $XY$, $\partial D$ smooth boundary of $D$. Given a real function  $f\in \mathcal{L}^{2}$, we would like to find a function $u$ such that 
\begin{eqnarray*}
-\Delta u&=&f\,\, \mbox{in} \,\, D \\
u&=&0\,\, \mbox{en} \,\, D
\end{eqnarray*} 
\textit{\textbf{Solution.}} \newline
Let  $T(z,w)=\left(\displaystyle \frac{d^{2}u }{dx^{2}} \right)z +\left(\displaystyle\frac{d^{2} u}{d y^{2}}\right)w$. 
Clearly $T$ is continuous and  $\ker T=\{ \left( -\displaystyle\frac{d^{2} u}{d y^{2}}, \displaystyle\frac{d^{2} u}{d x^{2}} \right) \}$ , $ \{ \ker T\}^{\perp}=\{ \left( \displaystyle\frac{d^{2} u}{d x^{2}}, \displaystyle\frac{d^{2} u}{d y^{2}} \right) \}$.
Therefore one has 
\begin{eqnarray}\label{ej2-1}
T(1,1)= \frac{d^{2} u}{d x^{2}}+\frac{d^{2} u}{d y^{2}}=-f
\end{eqnarray}
Moreover 
\begin{eqnarray}\label{ej2-2}
(1,1)=\lambda_{1}\left(-\frac{d^{2} u}{d y^{2}}, \frac{d^{2} u}{d x^{2}}\right)+\lambda_{2}\left(\frac{d^{2} u}{d x^{2}}, \frac{d^{2} u}{d y^{2}}\right)
\end{eqnarray}
from (\ref{ej2-2}) it is clear that  $\lambda_{1}=\lambda_{1}(x,y)$, $\lambda_{2}=\lambda_{2}(x,y)$. From the relations  (\ref{ej2-1}) and (\ref{ej2-2}) one has
 \begin{eqnarray}
 -f&=&\lambda_{2} \left[\left(\frac{d^{2} u}{d x^{2}}\right)^{2}+\left(\frac{d^{2} u}{d y^{2}}\right)^{2}  \right] \nonumber \\
 -f&=& \lambda_{2} \left[ f^{2}-2 \frac{d^{2} u}{d x^{2}} \left(-f-\frac{d^{2} u}{d x^{2}}  \right) \right]  \label{ej2-3}
 \end{eqnarray}
from (\ref{ej2-3}) one has
\begin{eqnarray}\label{ej2-4}
\frac{d^{2} u}{d x^{2}}&=& -\frac{f}{2} +\frac{1}{2} \sqrt{\frac{2 f}{\tau}-f^{2}}  \nonumber \\
\frac{d^{2} u}{d y^{2}}&=& -\frac{f}{2}-\frac{1}{2} \sqrt{\frac{2 f}{\tau}-t^{2}}
\end{eqnarray}
where $\tau (x,y)=-\lambda_{2}(x,y)$.
If $u$ is of the class $\mathcal{C}^{2}$, then from (\ref{ej2-4}) we have 
\begin{eqnarray}\label{ej2-5}
\iint_{D}\left(\frac{d^{2} u}{d x^{2}}-\frac{d^{2} u}{d y^{2}}\right) d x d y =\iint_{D} \sqrt{\frac{2 f}{\tau}-f^{2}} d x d y
\end{eqnarray}
Due to the Green's theorem one has that (\ref{ej2-5}) imply  $$\displaystyle\iint_{D} \sqrt{\frac{2 f}{\tau}-f^{2}} d x d y=0.$$
Then from (\ref{ej2-4}) one has
$$\frac{d^{2} u}{d x^{2}}=\frac{d^{2} u}{d y^{2}}=-\frac{f}{2}.$$
\end{example}
\begin{example} \cite{S}
Let us assume that there exist $A, B, C \in \mathbb{N}$ such that $A^{n}+B^{n}=C^{n}$, for some $n\in \mathbb{N} \wedge n\geq 3$, and $A,B,C$ are relatively prime numbers when taken two by two.
Define$F(x,y)=A^{n}x+B^{n}y$. It is clear that $\ker F= \left\{ \left(-B^{n}, A^{n} \right) \right\}$, $(\ker F )^{\perp}= \left\{ \left(A^{n}, B^{n} \right) \right\}.$
Thus, one has 
\begin{eqnarray}\label{ej3-1}
F(1,1)=A^{n}+B^{n}=C^{n}.
\end{eqnarray}
In addition one has
\begin{eqnarray}\label{ej3-2}
(1,1)=\lambda_{1}\left(-B^{n},A^{n} \right)+ \lambda_{2}\left(A^{n}, B^{n}\right) 
\end{eqnarray}
from (\ref{ej3-1}) and  (\ref{ej3-2})  one has
\begin{eqnarray}\label{ej3-3}
C^{n}=\lambda_{2}\left(A^{2n}+B^{2n}\right)=\lambda_{2}\left[C^{2 n}-2 A^{n} B^{n}\right] \nonumber\\
C^{n}=\lambda_{2} \left[C^{2n}-2A^{n}C^{n}+2A^{2n} \right] 
\end{eqnarray}
From the relation (\ref{ej3-3}) it is clear that  $\lambda_{2}=\frac{C^{n}}{m}$, where $C^{n}$ and $m$ are relatively prime integers and in addition  one gets the next relations
\begin{eqnarray}\label{ej3-4}
A^{n}=\frac{C^{n}}{2}+\sqrt{\frac{C^{n}}{2\lambda_{2}}- \frac{C^{2n}}{4}} \nonumber \\
B^{n}=\frac{C^{n}}{2}-\sqrt{\frac{C^{n}}{2\lambda_{2}}- \frac{C^{2n}}{4}} 
\end{eqnarray}
From (\ref{ej3-4}) one gets
\begin{eqnarray}\label{ej3-5}
A^{n}-B^{n}=\sqrt{\frac{2C^{n}}{\lambda_{2}}-C^{2n}}.
\end{eqnarray}
From (\ref{ej3-5}) one has
$$\frac{1}{\lambda_{2}^{2}}- \left(C^{n}-\frac{1}{\lambda_{2}}\right)^{2}=T^{2}, \,\, T\in \mathbb{N}.$$
Using  $\lambda_{2}=\frac{C^{n}}{m}$, in the last relation one gets
\begin{eqnarray}\label{ej3-6}
\left(T C^{n}\right)^{2}+\left(C^{2 n}-m\right)^{2}=m^{2}
\end{eqnarray}
Using the Fermat's theorem for $n=2$, from the relation (\ref{ej3-6}) one gets
$$m=N^{2}+M^{2}, \quad C^{2n}-m=M^{2}-N^{2}, \quad TC^{n}=2MN$$ where $M$ and $N$ are relatively prime integers.
So, one has
\begin{eqnarray*}
C^{2n}- \left( N^{2}+M^{2}\right)&=&M^{2}-N^{2}\\
C^{2n}=2M^{2}
\end{eqnarray*}
which is a contradiction.
\end{example}
\begin{example}
Let us consider the problem of finding the natural frequency and vibration modes of an extended  membrane on a bounded region  $D$ in $\mathbb{R}^{2}$ and boundary  $\partial D$. 
The small amplitude periodic vibrations are described by the equation  
\begin{eqnarray*}
\Delta u +\lambda u &=& 0, \,\,\mbox{en} \,\, D, \,\, \lambda >0.\\
u &=& 0, \mbox{en} \,\, \partial D.
\end{eqnarray*}
Following the ideas of the example \ref{EJEMPLO2} one has that for $u$ in the class $\mathcal{C}^{2}$; 
\begin{eqnarray}\label{eje4-1}
\dfrac{d^{2}u}{dx^{2}}= \dfrac{d^{2}}{dy^{2}}=\frac{-\lambda u}{2}.
\end{eqnarray}
As it is known a solution of (\ref{eje4-1}) is given by 
\begin{eqnarray}\label{eje4-2}
u(x,y)= \left(A\cos \sqrt{\frac{\lambda}{2}}x+ B\sin \sqrt{\frac{\lambda}{2}}x \right)\left(\widetilde{A}\cos \sqrt{\frac{\lambda}{2}}y+ \widetilde{B}\sin \sqrt{\frac{\lambda}{2} }y \right).
\end{eqnarray}
If one considers $D=[-L,L]\times [-L,L]$, then from the relation (\ref{eje4-2}) one has that for  $f(x)=A\cos \sqrt{\lambda}x+ B\sin \sqrt{\lambda}x $
\begin{eqnarray}
f(L)&=& A\cos \sqrt{\dfrac{\lambda}{2}}L+ B\sin \sqrt{\dfrac{\lambda}{2}}L=0  \label{eje4-3}\\
f(-L)&=& A\cos \sqrt{\dfrac{\lambda}{2}}L - B\sin \sqrt{\dfrac{\lambda}{2}}L=0 \label{eje4-4}.
\end{eqnarray}
From (\ref{eje4-3}) and  (\ref{eje4-4}) one has that  $A=0$. Therefore $\sqrt{\dfrac{\lambda}{2}}= \dfrac{k\pi}{L}$, $k\in \mathbb{N}$.
Analogously in (\ref{eje4-2}) one has $\widetilde{A}=0$.
Thus, one has the solution 
\begin{eqnarray}\label{eje4-5}
u_{k}(x,y)=C_{k} \left( \sin \frac{k\pi x}{L} \right)\left( \sin \frac{k\pi y}{L} \right), \,\, k\in \mathbb{N}.
\end{eqnarray}
Due to the superposition principle one has
$$u(x,y)=\sum_{k=1}^{+\infty} C_{k}\left(\sin \frac{k\pi x}{L} \right) \left(\sin \frac{k\pi y}{L} \right).$$
\end{example}
\begin{center} \subsection*{Acknowledgment} \end{center}
The author thanks CONCYTEC for  partial financial support and also thanks his family for providing him a nice  working environment. \\

\end{document}